\documentclass[11pt]{amsproc}

\usepackage{mathtext}
\usepackage[cp1251]{inputenc}

\usepackage{amsfonts}
\usepackage{amssymb}
\usepackage{textcomp}
\usepackage{latexsym,amsmath}
\usepackage{graphics}

\usepackage{bm}

\usepackage[dvips]{graphicx}
\usepackage{amsmath}
\usepackage{amssymb}
\usepackage{amsxtra}

\usepackage{epsfig}
\usepackage{epic}
\usepackage{eepic}
\usepackage{graphics}
\usepackage{graphicx}
\usepackage{subfigure}

\usepackage{mathrsfs}

\usepackage{caption}
\def\N{{{\Bbb N}}}
\def\Z{{{\Bbb Z}}}
\def\T{{{\Bbb T}}}
\def\R{{\Bbb R}}
\def\C{{\Bbb C}}
\def\l{{\lambda }}
\def\a{{\alpha }}

\def\a{{\alpha}}
\def\b{{\beta}}
\def\d{{\delta}}

\def\s{{\sigma}}
\def\vp{{\varphi}}

\def\g{{\gamma }}

\def\){\right)}
\def\({\left(}

\numberwithin{equation}{section}

\newtheorem{corollary}{Corollary}[section]
\newtheorem{lemma}{Lemma}[section]
\newtheorem{theorem}{Theorem}[section]
\newtheorem{proposition}{Proposition}[section]
\newtheorem{remark}{Remark}[section]
\newtheorem{definition}{Definition}[section]
\newtheorem{example}{Example}[section]

\newcommand{\h}{\widehat}
\newcommand{\w}{\widetilde}

\newcommand{\nul}{{\bf0}}
\newcommand{\rd}{{\mathbb R}^d}
\newcommand{\zd}{{\mathbb Z}^d}
\newcommand{\td}{{\mathbb T}^d}

\def\spec{\operatorname{spec}}
\renewcommand{\phi}{\varphi}

\par

\begin{document}

\title[Approximation properties of quasi-interpolation operators]{
Approximation properties of periodic multivariate  quasi-interpolation operators}

\author[Yurii
Kolomoitsev]{Yurii
Kolomoitsev$^{\text{a, 1, *}}$}
\address{Universit\"at zu L\"ubeck,
Institut f\"ur Mathematik,
Ratzeburger Allee 160,
23562 L\"ubeck}
\email{kolomoitsev@math.uni-luebeck.de}

\author[J\"urgen~Prestin]{J\"urgen~Prestin$^\text{a, 2}$}
\address{Universit\"at zu L\"ubeck,
Institut f\"ur Mathematik,
Ratzeburger Allee 160,
23562 L\"ubeck}
\email{prestin@math.uni-luebeck.de}

\thanks{$^\text{a}$Universit\"at zu L\"ubeck,
Institut f\"ur Mathematik,
Ratzeburger Allee 160,
23562 L\"ubeck, Germany}


\thanks{$^1$Supported by DFG project KO 5804/1-1.}

\thanks{$^2$Supported by Volkswagen Foundation.}

\thanks{$^*$Corresponding author}

\thanks{E-mail address: kolomoitsev@math.uni-luebeck.de}

\date{\today}

\begin{abstract}
We study approximation properties of general multivariate periodic quasi-interpolation operators, which are generated by distributions/functions $\widetilde{\varphi}_j$ and trigonometric polynomials $\varphi_j$. The class of such operators includes classical interpolation polynomials ($\widetilde\varphi_j$ is the Dirac delta function), Kantorovich-type operators ($\widetilde\varphi_j$ is a characteristic function), scaling expansions associated with wavelet constructions, and others.  Under different compatibility conditions on $\widetilde\varphi_j$ and $\varphi_j$, we obtain upper and lower bound estimates for the $L_p$-error of approximation by quasi-interpolation  operators in terms of the best and best one-sided approximation, classical and fractional moduli of smoothness, $K$-functionals, and other terms.
\end{abstract}

\maketitle

\section{Introduction}

Quasi-interpolation operators are among the most important mathematical tools
in many branches of science and engineering. They play a crucial role as a connecting
link between continuous-time and discrete-time signals. For proper application of quasi-interpolation ope\-rators, it is very important to know the quality of approximation of functions by such operators in various settings.
Recall that in the non-periodic case, quasi-interpolation operators, which are also often called quasi-projection operators, can be defined by
\begin{equation}\label{my00}
   \sum_{k\in\Z^d} m^{j} \langle f,\w\vp(M^j\cdot-k)\rangle \vp(M^j\cdot-k),
\end{equation}
where $\phi$ is a function and $\w\phi$ is a distribution or a function, $\langle f,\w\vp(M^j\cdot-k)\rangle$ is an appropriate functional, $M$ is a dilation matrix, and  $m=|\det M|$.
The class of these operators is very large. For example, if $\w\vp$ is the Dirac delta-function, operators~\eqref{my00} represent classical sampling expansions  (see, e.g.,~\cite{Unser, Butz4, Butz5, JZ, KKS, KS20a}); if $\w\vp$ is a characteristic function of a certain bounded set, we obtain the so-called Kantorovich-type operators and their generalization  (see, e.g.,~\cite{BBSV, OT15,CV2, VZ2, KS3, KS20b}); under particular conditions on $\vp$ and $\w\vp$, the class of operators~\eqref{my00} includes scaling expansions associated with wavelet constructions  (see, e.g.,~\cite{Jia2,  BDR, v58, KS0, Sk1}) and other types of operators.

In this paper, we study a periodic counterpart of~\eqref{my00}, which can be defined in the following way
\begin{equation}\label{my01}
Q_j(f,\phi_j,\w\phi_j)  = \frac{1}{m^j} \sum\limits_{k} {\w \phi_j}*f(M^{-j}k) \phi_j (\cdot - M^{-j}k),
\end{equation}
where the sum over $k$ is finite, $\phi_j$ is a trigonometric polynomial, and ${{\w \phi_j}}*f$ is a certain bounded function associated with the distribution/function $\w\vp_j$  (see Section~2 for details).

Similar to the non-periodic case, approximation properties of operators~\eqref{my01}
have also been intensively studied by many mathematicians (see, e.g.,~\cite{H, Unser2002, KKS2, PrXu, SS99b, Sp00, Sz} and the references therein). It turns out that in the
periodic case, such operators have been considered mainly in the form of sampling or interpolating-type
operators (i.e., $\w\vp_j$ is the periodic Delta function) given by
\begin{equation}\label{my02}
I_j(f,\phi_j)  = \frac{1}{m^j} \sum\limits_{k} f(M^{-j}k) \phi_j (\cdot - M^{-j}k),
\end{equation}
where, usually, $\phi_j$ is a so-called fundamental interpolant, e.g., the Dirichlet or de la Vall\'ee-Poussin kernels, or
periodic $B$-splines. At the same time, general periodic quasi-interpolation operators of type~\eqref{my01}
have been studied only in a few works. In particular, the general case of operators~\eqref{my01} with some particular class
of linear functionals instead of $f(M^{-j}k)$ was studied in~\cite{Unser2002} and in the recent paper~\cite{KKS2}.

The estimation of the $L_p$-error of approximation by interpolation operators~\eqref{my02}, in which $\vp_j$ is the Dirichlet kernel was studied in~\cite{H}. A more general case of Hermite-type interpolation was considered in~\cite{PrXu}.
In the above mentioned two papers, the estimates of the  error were given in terms of the best one-sided approximation by trigonometric polynomials and in terms of  the $\tau$-modulus of smoothness of arbitrary integer order.  Approximation properties of  operators~\eqref{my02} for various trigonometric polynomials $\vp_j$ (the so-called methods of summation of the discrete Fourier series) were considered  in~\cite{Sz} and~\cite{SzV}, in which the error estimates were investigated in the uniform norm.
In the papers~\cite{SS99b} and \cite{Sp00}, the introduction of the periodic Strang-Fix conditions as well as their  different  modifications enabled the development of a unified approach to error estimates of periodic interpolation for functions from the Sobolev spaces and other function spaces.
Some estimates of the $L_p$-error of approximation by operators~\eqref{my02} for functions from Nikol'skij-Besov spaces were derived in~\cite{SS99}.

The goal of this paper is to estimate the $L_p$-error of approximation of a given function $f$,  from above and below, by quasi-interpolation operators $Q_j(f,\vp_j,\w\vp_j)$ for a wide range of distributions/functions $\w\vp_j$ and trigonometric polynomials $\vp_j$.
Under different compatibility conditions on $\vp_j$ and $\w\vp_j$ related in some sense to the Strang-Fix conditions, we obtain  estimates for the error of approximation in terms of the best and best one-sided approximation (see the definition in Section~\ref{sec2}), classical and fractional moduli of smoothness, $K$-functionals, and other terms.
We pay a special attention to the case $\w\vp_j\in L_q$, for example, $\w\vp_j$ is a normalized characteristic function, which provides Kantorovich-type operators. In particular, we show that if $\vp_j=\mathscr{D}_{2^{j}}$ is the  Dirichlet kernel and $f\in L_p[-\frac12,\frac12]$, $1<p<\infty$, $\s\in (0,1/2]$, then (see Example~\ref{e4})
\begin{equation}\label{intr1}
  \bigg\Vert f- \sum_{k=-2^{j-1}}^{2^{j-1}-1} \frac1{2\s}\int_{-2^{-j}\s}^{2^{-j}\s} f\(t+{2^{-j}k}\){\rm d}t\, \mathscr{D}_{2^{j}}\(\cdot-{2^{-j}k}\)\bigg\Vert_p\asymp \omega_2(f,2^{-j})_p,
\end{equation}
where $\omega_2(f,2^{-j})_p$ is the classical modulus of smoothness of second order.
At the same time, if $\vp_j(x)=\mathscr{D}_{2^{j},\s}^\chi(x)=\sum_{\ell=-2^{j-1}}^{2^{j-1}-1}\frac{\pi \s 2^{-j+1}\ell}{\sin \pi \s 2^{-j+1}\ell} {\rm e}^{2\pi {\rm i} \ell x}$ and $1<p<\infty$, then (see Example~\ref{e2})
\begin{equation}\label{intr2}
  \bigg\Vert f- \sum_{k=-2^{j-1}}^{2^{j-1}-1} \frac1{2\s}\int_{-2^{-j}\s}^{2^{-j}\s} f\(t+{2^{-j}}k\){\rm d}t\, \mathscr{D}_{2^{j},\s}^\chi \(\cdot-{2^{-j}k}\)\bigg\Vert_p\asymp E_{2^{j}}(f)_p,
\end{equation}
where $E_{2^{j}}(f)_p$ is the $L_p$-error of the best approximation of $f$ by trigonometric polynomials with frequencies in $[-2^{j-1},2^{j-1})$.
In the above relations~\eqref{intr1} and~\eqref{intr2}, the notation $\asymp$ denotes the two-sided inequality with positive constants that do not dependent on~$f$ and~$j$.

The paper is organized as follows: in
Section~2 we introduce basic notations, provide essential facts, and define the quasi-interpolation operator $Q_j(f,\vp_j,\w\vp_j)$.
Section~3 is devoted to auxiliary results. In this section, we obtain general upper estimates of the $L_p$-error for $Q_j(f,\vp_j,\w\vp_j)$ and give auxiliary lemmas.
In Section~4 we prove the main results. In Subsection~4.1, under strong compatibility conditions on $\vp_j$ and $\w\vp_j$, we estimate the $L_p$-error for operators~\eqref{my01} in terms of best approximation by trigonometric polynomials. In Subsection~4.2 we give two-sided estimates of the approximation error $\Vert f-Q_j(f,\vp_j,\w\vp_j)\Vert_p$ in terms of classical and fractional moduli of smoothness and $K$-functionals. In Subsection~4.3 we spe\-cify some error estimates from the previous section for functions $f$ belonging to Besov-type spaces.

\section{Basic notation}\label{sec2}

We use the standard multi-index notations.
    Let $\N$ be the set of positive integers, $\rd$ be the $d$-dimensional Euclidean space,
    $\Z^d$ be the integer lattice  in $\R^d$,
    $\T^d=\R^d\slash\Z^d$ be the $d$-dimensional torus.
    Further, let  $x = (x_1,\dots, x_d)^{\rm T}$ and
    $y =(y_1,\dots, y_d)^{\rm T}$ be column vectors in $\rd$.
    Then $(x, y):=x_1y_1+\dots+x_dy_d$,
    $|x| := \sqrt {(x, x)}$; $\nul=(0,\dots, 0)^T\in \rd$;  		
		$\Z_+^d:=\{x\in\zd:~x_k\geq~{0}, k=1,\dots, d\}.$
		If $\alpha\in\zd_+$,  we set
    $[\alpha]=\sum_{k=1}^d \alpha_k$,
$D^{\alpha}f=\frac{\partial^{[\alpha]} f}{\partial^{\alpha_1}x_1\dots
	\partial^{\alpha_d}x_d}$.

We denote by $c$ and $C$ some positive constants depending on the
indicated parameters.  By these letters we also denote some positive constants that are independent
of the function $f$ and the parameter $j$.

 We  use the notation $L_p$ for the space  $L_p(\td)$ with the usual norm
$$
\|f\|_p = \Big(\int_{\td}|f(x)|^p {\rm d}x\Big)^{1/p}\quad \text{for}\quad  1\le p<\infty
$$
and
$$
\|f\|_\infty={{\rm ess} \sup}_{x\in \T^d} |f(x)| \quad \text{for}\quad  p=\infty.
$$
When $p=\infty$, we replace  $L_\infty$ with $C(\T^d)$.  By $B=B(\T^d)$ we denote the space of all bounded measurable functions on $\T^d$.

If $f\in L_1(\T^d)$, then
$$
\h f(k)=\int_{\T^d} f(x){\rm e}^{-2\pi {\rm i}(k,x)}{\rm d}x,\quad k\in\zd,
$$
denotes the $k$-th Fourier coefficient of $f$.
The Fourier transform of  $f\in L_1(\R^d)$ is defined by $\mathcal{F}(f)(\xi)=\int_{\R^d} f(x)e^{-2\pi i(x,\xi)}{\rm d}x$.

Let ${\mathcal{D}} = C^{\infty}(\T^d)$ be the space of infinitely differentiable functions on $\R^d$ that are periodic with period 1.
The linear space of periodic distributions (continuous linear functionals on ${\mathcal{D}}$) is denoted by ${\mathcal{D}}'$.
It is known (see, e.g.,~\cite[p. 144]{Triebel}) that any periodic distribution ${\w\vp}$ can be expanded in a weakly convergent (in ${\mathcal{D}}'$) Fourier series
\begin{equation}
{\w\vp}(x) = \sum_{k\in\zd} \h{\w\vp}(k) {\rm e}^{2\pi {\rm i} (k, x)},
\label{fPDisrt}
\end{equation}
where the sequence $\{\h{\w\vp}(k)\}_k$ has at most polynomial growth.
Also, conversely, for any sequence $\{\h{\w\vp}(k)\}_k$ of at most polynomial growth the series on the right-hand side of~(\ref{fPDisrt}) converges weakly to a periodic distribution.
The numbers $\h{\w\vp}(k)$ are called the Fourier coefficients of a periodic distribution ${\w\vp}$ and
$\h{\w\vp}(k) = \langle {\rm e}^{-2\pi {\rm i} (k, \cdot)}, {\w\vp}\rangle $.

In what follows, $M={\rm diag} (m_1,m_2,\dots,m_d)$ is a diagonal dilation matrix, $m_j$ is an integer with $|m_j|>1$, $m:=|\det M|$,   $D(M) : = (M [-1/2,1/2)^d) \cap \zd$.

For a given matrix $M$, we will use the following set of trigonometric polynomials:
$$
\mathcal{T}_{M}:=\{T\,:\, \spec T\subset D(M)\}.
$$
The $L_p$-error of the best approximation of $f\in L_p$ by trigonometric polynomials $T\in \mathcal{T}_{M}$ is denoted by
$$
E_{M}(f)_p:=\inf\left\{\Vert f-T\Vert_p\,:\, T\in \mathcal{T}_{M}\right\}.
$$
The $L_p$-error of the best one-sided approximation of $f\in B$ is given by
$$
\w E_{M}(f)_p:=\inf\left\{\Vert t-T\Vert_p\,:\, t,T\in \mathcal{T}_{M}, \quad t(x)\le f(x)\le T(x)\quad\text{for all}\quad x \in \T^d\right\}.
$$
Note that for $p=\infty$ the error of the best one-sided approximation coincides up to a constant with the error of the unrestricted best approximation $E_{M}(f)_p$, see, e.g.,~\cite[p.~163]{SP}.

For a sequence $\{a_k\}_{k\in D(M)}\in \C$, we denote
$$
\left\Vert \{a_k\}_{k}\right\Vert_{\ell_{p,M}}:=\left\{
                                                 \begin{array}{ll}
                                                  \displaystyle\bigg(\frac1{m}\sum\limits_{k\in D(M)}|a_k|^p\bigg)^\frac1p, & \hbox{if $1\le p<\infty$,} \\
                                                   \displaystyle\sup\limits_{k\in D(M)}|a_k|, & \hbox{if $p=\infty$.}
                                                 \end{array}
                                               \right.
$$

In this paper, we will use the following notation for the rectangular partial sums of the Fourier series and the de la Vall\'ee Poussin means of $f$:
$$
S_{M}f(x):=\sum_{k\in D(M)} \h f(k) {\rm e}^{2\pi {\rm i}(k,x)},
$$
$$
V_{M}f(x):=\sum_{k\in D(M)} v(M^{-1}k) \h f(k){\rm e}^{2\pi {\rm i}(k,x)},
$$
where $v\in C^\infty (\R^d)$, $v(\xi)=1$ for $\xi\in [-1/4,1/4)^d$ and $v(\xi)=0$ for $\xi\not\in [-3/8,3/8)^d$. Recall the following well-known inequalities (see, e.g.,~\cite[2.1.6, 2.4.5, and 4.1.1]{TB}):
\begin{equation}\label{se}
\Vert f-S_M f\Vert_p\le c(p,d)E_{M}(f)_p,\quad 1<p<\infty,
\end{equation}
\begin{equation}\label{ve}
  \Vert f-V_M f\Vert_p\le \(1+\|v\|_{L_1(\R^d)}\)E_{\frac 12M}(f)_p\le c(d)E_{\frac 12M}(f)_p,\quad 1\le p\le \infty.
\end{equation}

The Dirichlet kernel with respect to the matrix $M$ is defined by
$$
\mathscr{D}_{M}(x)=\sum_{k\in D(M)} {\rm e}^{2\pi {\rm i}(k,x)}.
$$

Let $\vp$ be a a trigonometric polynomial and $f\in L_p$, $1\le p\le \infty$. Denote
$$
K_{\phi,p} :=\sup_{\Vert f\Vert_p\le 1}\Vert {\phi}*f\Vert_p.
$$
Note that (see, e.g.,~\cite[Ch.~8]{TB}) if
$\h{\phi_j}(\xi)=\chi_{M^j [-\frac12,\frac12)^d}(\xi)$, where $\chi_G$ denote the characteristic function of the set $G$, then ${\phi_j}*f=S_{M^j} f$ and
$$
K_{\phi_j,p} \asymp \left\{
                                 \begin{array}{ll}
                                   1, & \hbox{$1<p<\infty$,} \\
                                   j^d, & \hbox{$p=1$ or $\infty$.}
                                 \end{array}
                               \right.
$$
The averaging operator with respect to the matrix $M$ is defined by
$$
{\rm Avg}_M f(x)=m^{-1} \int_{M[-\frac12,\frac12)^d} f(t+x){\rm d}t.
$$

\begin{definition}\label{def1}
Let $\w\vp \in \mathcal{D}'$ and $1\le p\le \infty$. We will say that a function $f$ belongs to the class $B_{\w\vp,p}$ if $f\in L_p$
and 
$$
\sum_{\ell\in\zd}
 {\h{\w\phi}(\ell)} \h f(\ell)    {\rm e}^{2 \pi {\rm i} (\ell,x)}
$$
is a Fourier series of a certain bounded function, which we denote by ${{\w \phi}}*f$.
\end{definition}


Typical examples of $B_{\w\vp,p}$ are the following:
1) if $\w\vp$ is  a finite complex-valued Borel measure on $\T^d$ and $p=\infty$, then $B_{\w\vp,p}=B$, see, e.g.,~\cite[7.1.4]{TB};
2) if $\w\vp \in L_q$, $1/p+1/q=1$, then by Young's convolution inequality, we have that $B_{\w\vp,p}=L_p$.


Now, let us introduce the main object of this paper.  Let $j\in \N$, $\w\vp_j \in \mathcal{D}'$, $\vp_j \in L_p$, and $f\in B_{\w\vp_j,p}$ be given.
We define the general multivariate periodic quasi-interpolation operator by
\begin{equation}\label{oper}
  Q_j(f,\vp_j,\w\vp_j)(x)=\frac{1}{m^j} \sum\limits_{k\in D(M^j)} {\w \phi_j}*f(M^{-j}k) \phi_j (x - M^{-j}k).
\end{equation}

Note that for functions $f$ from some special Wiener and Besov classes, similar quasi-interpolation operators have been recently studied in~\cite{KKS2}. Particularly, in terms of decay of the Fourier coefficients of $f$, there were obtained several types of estimates of approximation by operators~\eqref{oper} in the Wiener-type spaces and the spaces $L_p(\T^d)$ with $2\le p\le \infty$. In the present paper, we essentially extend and improve the results given in~\cite{KKS2} in several directions. First of all,
using an approach based on the best one-sided approximations and Fourier multipliers, we obtain error estimates in $L_p(\T^d)$ for all $1\le p\le \infty$. Second, using new type of compatibility conditions for $\vp_j$ and $\w\vp_j$, we give the corresponding error estimates in terms of classical and fractional moduli of smoothness and $K$-functionals, which are commonly used in approximation theory and in most cases provide sharper estimates than those given in~\cite{KKS2} in terms of the Fourier coefficients of $f$. Third, together with estimates from above of the $L_p$-error of approximation, we obtain also the estimates from below, which show the sharpness of our results for particular classes of quasi-interpolation operators.

\bigskip
\bigskip

\section{Auxiliary results}

The next lemma is one of the main auxiliary results in this paper.

\begin{lemma}\label{thKS}
Let $1\le p\le\infty$, $1/p+1/q=1$, $\d\in (0,1]$,  and $j\in \N$. Suppose that $\w\phi_j\in \mathcal{D}'$ and
$\phi_j \in \mathcal{T}_{M^j}$.
Then, for any $f\in B_{\w\vp_j, p}$, we have
\begin{equation*}
\begin{split}
     \Vert f - Q_j(f,\phi_j,\w\phi_j) \Vert_p\le C \Big(\Vert {\psi_j}*T_j\Vert_{p}&+E_{\d M^j}(f)_p\\
     &+
     K_{\phi_j,q} \big(\w E_{M^j}({{\w\phi_j}}*f)_p+\Vert {{\w\phi_j}}*f-{{\w\phi_j}}*T_j\Vert_p\big)\Big),
\end{split}
\end{equation*}
where
\begin{equation}\label{psi1}
  \psi_j(x)=\sum_{\ell\in D(M^j)} \(1-\h{\phi_j}(\ell)\h{\w\phi_j}(\ell)\){\rm e}^{2\pi{\rm i}(\ell,x)},
\end{equation}
the polynomial $T_j\in \mathcal{T}_{M^j}$ is such that $\Vert f-T_j\Vert_p\le c(d,p,\d)E_{\d M^j}(f)_p$,
and the constant $C$ does not depend on $f$ and $j$.
\end{lemma}

Before proving Lemma~\ref{thKS}, we give one simple corollary of Lemma~\ref{thKS} for the partial sums of the Fourier series $S_{M^j} f$ and the de la Vall\'ee Poussin means $V_{M^j} f$.

\begin{corollary}\label{cor1thKS}
Under the conditions of Lemma~\ref{cor1thKS}, we have:

a) if $1<p<\infty$, then
\begin{equation}\label{KS000_c1}
\begin{split}
     \Vert f - Q_j(f,\phi_j,\w\phi_j) \Vert_p\le C \Big(\Vert {\psi_j}*S_{M^j}f\Vert_{p}+E_{M^j}(f)_p+
     K_{\phi_j,q} \w E_{M^j}({{\w\phi_j}}*f)_p\Big),
\end{split}
\end{equation}

b) if $1\le p\le\infty$, then
\begin{equation}\label{KS000_c2}
\begin{split}
     \Vert f - Q_j(f,\phi_j,\w\phi_j) \Vert_p\le C \Big(\Vert {\psi_j}*V_{M^j}f\Vert_{p}+E_{\frac12M^j}(f)_p+
     K_{\phi_j,q} \w E_{\frac12 M^j}({{\w\phi_j}}*f)_p\Big),
\end{split}
\end{equation}
where the constant $C$ does not depend on $f$ and $j$ and the function $\psi_j$ is given by~\eqref{psi1}.
\end{corollary}

\begin{proof}
The inequalities~\eqref{KS000_c1} and~\eqref{KS000_c2} can be obtained repeating the proof of Lemma~\ref{thKS} presented below by taking $T_j=S_{M^j} f$ in the case $1<p<\infty$ and $T_j=V_{M^j} f$ in the case $1\le p\le \infty$.
We need also to use~\eqref{se}, \eqref{ve}, and the following simple inequalities
\begin{equation}\label{zvezzvez}
  \begin{split}
    \Vert {{\w\phi_j}}*f-{{\w\phi_j}}*V_{M^j} f\Vert_p&=\Vert {{\w\phi_j}}*f-V_{M^j} ({{\w\phi_j}}*f)\Vert_p\\
&\le C E_{\frac12 M^j}({{\w\phi_j}}*f)_p\le C \w E_{\frac12 M^j}({{\w\phi_j}}*f)_p.
  \end{split}
\end{equation}
\end{proof}

To prove Lemma~\ref{thKS}, we will use a standard Marcinkiewicz-Zygmund inequality for multivariate trigonometric polynomials given in the following lemma. Its proof follows easily from the corresponding one-dimensional result, see, e.g.,~\cite{LMN}.

\begin{lemma}\label{lemMZ}
Let $1\le p\le\infty$, $j\in \N$, and $T_j\in \mathcal{T}_{M^j}$. Then
\begin{equation*}
  \left\Vert \{ T_j(M^{-j}k) \}_{k} \right\Vert_{\ell_{p,M^j}}\le c(d,p)\Vert T_j\Vert_p.
\end{equation*}
\end{lemma}

The next lemma was proved in~\cite[Lemma16]{KKS2}.

\begin{lemma}\label{lemKK1} 
Let $1\le p\le \infty$, $1/p+1/q=1$, $j\in \N$, $\{a_k\}_{k}\in \mathbb{C}$, and $\phi_j \in \mathcal{T}_{M^j}$. Then
\begin{equation*}
  \bigg\Vert \frac1{m^j}\sum_{k\in D(M^j)} a_k \phi_j(\cdot-M^{-j}k)\bigg\Vert_p
\le  C K_{\phi_j,q} \left\Vert \{a_k\}_{k}\right\Vert_{\ell_{p,M^j}},
\end{equation*}
where the constant $C$ does not depend on $j$ and $\{a_k\}$.
\end{lemma}


\begin{proof}[Proof of Lemma~\ref{thKS}.]  We consider only the case $1\le p<\infty$. The case $p=\infty$ can be treated similarly.
We have
\begin{equation}\label{I123}
  \begin{split}
      &\bigg\Vert f - \frac1{m^j}\sum\limits_{k\in D(M^{j})} {\w \phi_j}*f(M^{-j}k) \phi_j (\cdot - M^{-j}k)\bigg\Vert_p\\
      &\le \Vert f-T_j\Vert_p+\bigg\Vert T_j-\frac1{m^j}\sum\limits_{k\in D(M^{j})} {{\w \phi_j}}*T_j(M^{-j}k) \phi_j (\cdot - M^{-j}k)\bigg\Vert_p\\
      &\quad+\bigg\Vert \frac1{m^j}\sum\limits_{k\in D(M^{j})} \({\w \phi_j}*f(M^{-j}k)-{\w \phi_j}*T_j(M^{-j}k)\) \phi_j (\cdot - M^{-j}k)\bigg\Vert_p:=I_1+I_2+I_3.
   \end{split}
\end{equation}

First, we consider $I_2$. We have
\begin{equation}\label{zvezda1ra}
  \begin{split}
      &T_j(x)-\frac1{m^j}\sum\limits_{k\in D(M^{j})} {{\w \phi_j}}*T_j(M^{-j}k) \phi_j (x - M^{-j}k)\\
      &=\sum_{\ell \in D(M^j)} \(\h{T_j}(\ell)-\frac{\h{\phi_j}(\ell)}{m^j} \sum_{k\in D(M^j)} {\w{\phi}_j}*T_j(M^{-j}k) {\rm e}^{-2\pi {\rm i}(\ell, M^{-j}k)}\) {\rm e}^{2\pi {\rm i}(\ell,x)}\\
      &=\sum_{\ell \in D(M^j)} \(\h{T_j}(\ell)-{\h{\phi_j}(\ell)}\sum_{\nu\in D(M^j)} \h{\w \phi_j}(\nu)\h{T_j}(\nu) \frac1{m^j} \sum_{k\in D(M^j)} {\rm e}^{2\pi {\rm i}(\nu-\ell, M^{-j}k)}\) {\rm e}^{2\pi {\rm i}(\ell,x)}\\
      &=\sum_{\ell \in D(M^j)} \(\h{T_j}(\ell)-\h{\phi_j}(\ell)\h{\w \phi_j}(\ell)\h{T_j}(\ell)\) {\rm e}^{2\pi {\rm i}(\ell,x)}={\psi_j}*T_j(x),
   \end{split}
\end{equation}
which implies that
\begin{equation}\label{I2}
  I_2=\Vert {\psi_j}*T_j\Vert_p.
\end{equation}

Consider $I_3$.
Let $u_j, U_j \in \mathcal{T}_{M^j}$ be such that $u_j(x)\le{{\w\phi_j}}*f(x)\le U_j(x)$ for all $x\in \T^d$ and $\|u_j-U_j\|_p\le 2\w E_{M^j}({{\w\phi_j}}*f)_p$.
Then, using Lemmas~\ref{lemKK1}   and~\ref{lemMZ}, we derive
\begin{equation}\label{I3}
  \begin{split}
        I_3&\le C K_{\phi_j,q} \bigg(\frac1{m^j} \sum_{k\in D(M^j)}|{{\w\phi_j}}*f(M^{-j} k)-{{\w\phi_j}}*T_j(M^{-j} k)|^p\bigg)^\frac1p\\
        &\le C K_{\phi_j,q}\Bigg(\bigg(\frac1{m^j} \sum_{k\in D(M^j)}|U_j(M^{-j} k)-{{\w\phi_j}}*T_j(M^{-j} k)|^p\bigg)^\frac1p\\
        &\qquad\qquad\qquad\qquad+\bigg(\frac1{m^j} \sum_{k\in D(M^j)}|U_j(M^{-j} k)-{{\w\phi_j}}*f(M^{-j} k)|^p\bigg)^\frac1p\Bigg)\\
        &\le C K_{\phi_j,q} \Bigg(\Vert U_j-{{\w\phi_j}}*T_j\Vert_p+\bigg(\frac1{m^j} \sum_{k\in D(M^j)}|U_j(M^{-j} k)-u_j(M^{-j} k)|^p\bigg)^\frac1p\Bigg)\\
        &\le C K_{\phi_j,q} \(\Vert U_j-{{\w\phi_j}}*T_j\Vert_p+\|U_j-u_j\|_p\)\\
        &\le C K_{\phi_j,q} \(\| U_j-{{\w\phi_j}}*f\|_p+\|U_j-u_j\|_p+\|{{\w\phi_j}}*f-{{\w\phi_j}}*T_j\|_p\)\\
        &\le C K_{\phi_j,q} \(\w E_{M^j}({{\w\phi_j}}*f)_p+\|{{\w\phi_j}}*f -{{\w\phi_j}}*T_j\|_p\).
   \end{split}
\end{equation}

Finally, combining~\eqref{I123}, \eqref{I2}, and~\eqref{I3}, we prove the lemma.
\end{proof}


In Lemma~\ref{thKS}, the error estimate was given in terms of the best one-sided approximation $\w E_{M^j}({{\w\phi_j}}*f)_p$ for the function $f\in B_{\w\vp_j,p}$.
Under more restrictive conditions on the function $\w\vp_j$, we can take $B_{\w\vp_j,p}=L_p$ and replace the best one-sided approximation with the unrestricted best approximation.
For this, we will use the following special norms for a function $\w\vp_j \in L_q$, $j\in \N$:
$$
\Vert \w\phi_j\Vert_{\mathcal{L}_{q,j}}:=\(m^j\int_{M^{-j}\T^d} \bigg(\frac1{m^j}\sum_{k\in D(M^j)} |\w\phi_j(x-M^{-j}k)| \bigg)^q {\rm d}x\)^\frac1q\quad\text{if}\quad 1\le q<\infty
$$
and
$$
\Vert \w\phi_j\Vert_{\mathcal{L}_{\infty,j}}:=\frac1{m^j}\sup_{x\in \R^d}\sum_{k\in D(M^j)} |\w\phi_j(x-M^{-j}k)| \quad\text{if}\quad q=\infty.
$$

We have the following improvement of Lemma~\ref{thKS} for $\w\vp_j \in L_q$:

\begin{lemma}\label{thKK}
Let $1\le p\le\infty$, $1/p+1/q=1$, $\d\in (0,1]$,  and $j\in \N$. Suppose that $\w\phi_j\in L_q$ and
$\phi_j\in \mathcal{T}_{M^j}$.
Then, for any $f\in L_p$, we have
\begin{equation*}
\begin{split}
     \Vert f - Q_j(f,\phi_j,\w\phi_j) \Vert_p&\le C \Big(\Vert {\psi_j}*T_j\Vert_{p}+(1+K_{\phi_j,q} \Vert \w\phi_j \Vert_{\mathcal{L}_{q,j}})E_{\d M^j}(f)_p\Big),
\end{split}
\end{equation*}
where $\psi_j$ is given by~\eqref{psi1}, the polynomial $T_j\in \mathcal{T}_{M^j}$ is such that $\Vert f-T_j\Vert_p\le c(d,p,\d)E_{\d M^j}(f)_p$,
and the constant $C$ does not depend on $f$ and $j$.
\end{lemma}

\medskip

The proof of Lemma~\ref{thKK} is based on the following result  (see Lemma~17 in~\cite{KKS2}):

\begin{lemma}\label{lemKK2-}
Let $1\le p\le \infty$, $1/p+1/q=1$, $j\in \N$, and $\w\phi_j \in L_{q}$. Then, for any $f\in L_p$, we have
\begin{equation*}
 \left\Vert \left\{{\w\phi_j}*f(M^{-j}k) \right\}_{k} \right\Vert_{\ell_{p,M^j}}\le \Vert \w\phi_j \Vert_{\mathcal{L}_{q,j}}\Vert f\Vert_p.
\end{equation*}
\end{lemma}


\begin{proof}[Proof of Lemma~\ref{thKK}.] The proof is similar to the proof of Lemma~\ref{thKS}. It is sufficient to use inequalities~\eqref{I123} and~\eqref{I2} as well as the following estimate
\begin{equation}\label{zvezzvez2}
  \begin{split}
      I_3&\le C K_{\phi_j,q}\(\frac1{m^j}\sum\limits_{k\in D(M^{j})} |{\w\phi_j}*(f-T_j)(M^{-j}k)|^p\)^\frac1p \\
         &\le C K_{\phi_j,q} \Vert \w\phi_j \Vert_{\mathcal{L}_{q,j}}\Vert f-T_j\Vert_p
   \end{split}
\end{equation}
instead of inequality~\eqref{I3}. The above estimate easily follows from  Lemmas~\ref{lemKK1} and~\ref{lemKK2-}.
\end{proof}

\section{Main results}

\subsection{Estimates of approximation in terms of best approximation}

In this subsection, we give an explicit form of the error estimates from Lemmas~\ref{thKS} and~\ref{thKK} in the case of the so-called strictly compatible functions/distributions $\phi_j$ and $\w\phi_j$.

\begin{theorem}\label{cor1}
Let $1\le p\le\infty$, $1/p+1/q=1$, $0<\d\le \rho\le 1$,  and $j\in \N$. Suppose that $\w\phi_j\in \mathcal{D}'$ and
$\phi_j \in \mathcal{T}_{M^j}$ are such that
\begin{equation}\label{sc}
   \h{\phi_j}(k)\h{\w\phi_j}(k)=1\quad \text{for all}\quad k\in D(\rho M^j).
\end{equation}
Then, for any $f\in  B_{\w\vp_j, p}$, we have
\begin{equation}\label{KS000}
\begin{split}
     \Vert f - Q_j(f,\phi_j,\w\phi_j) \Vert_p&\le C \(E_{\delta M^j}(f)_p+K_{\phi_j,q} \(\w E_{M^j}({{\w\phi_j}}*f)_p+\Vert {{\w\phi_j}}*(f-T_j)\Vert_p\)\),\\
\end{split}
\end{equation}
where $T_j\in \mathcal{T}_{\rho M^j}$ is such that $\Vert f-T_j\Vert_p\le c(d,p,\d)E_{\delta M^j}(f)_p$; if, additionally, $\w\phi_j\in L_q$, then, for any $f\in L_p$, we have
\begin{equation}\label{KS000Kant}
  \Vert f - Q_j(f,\phi_j,\w\phi_j) \Vert_p\le C (1+K_{\phi_j,q}\Vert \w\phi_j \Vert_{\mathcal{L}_{q,j}}) E_{\delta M^j}(f)_p,
\end{equation}
where the constant $C$ does not depend on $f$ and~$j$.
\end{theorem}

Note that inequality~\eqref{KS000Kant} was earlier obtained in~\cite{KKS2}.

\begin{proof}
To prove the theorem, it is enough to use Lemmas~\ref{thKS},~\ref{thKK} and to take into account that  $\Vert {\psi_j}*T_j\Vert_{p}=0$ and all estimates in the proof of Lemma~\ref{thKS} remain the same for $T_j\in \mathcal{T}_{\rho M^j}$.
\end{proof}

Similarly to Corollary~\ref{cor1thKS}, we derive the following result:

\begin{corollary}\label{cor1thKS+1}
Under the conditions of Theorem~\ref{cor1}, we have that inequality~\eqref{KS000} can be replaced by
\begin{equation*}
\begin{split}
     \Vert f - Q_j(f,\phi_j,\w\phi_j) \Vert_p\le C \Big(E_{\d M^j}(f)_p+
     K_{\phi_j,q} \w E_{\d M^j}({{\w\phi_j}}*f)_p\Big),
\end{split}
\end{equation*}
where $\d<\rho$ if $p=1,\infty$ and 
the constant $C$ does not depend on $f$ and $j$.
\end{corollary}

\begin{example}\label{e1}
If $\w\vp_j$ is the periodic Dirac delta function for all $j\in \N$ and $\vp_j=\mathscr{D}_{M^j}$ is the Dirichlet kernel,
then equality~\eqref{sc} obviously holds with $\rho=\d=1$ and inequality~\eqref{KS000}  implies the following well-known error estimate for the corresponding interpolation operator  (cf.~\cite[Corollary~3]{H}):
\begin{equation*}
 \bigg\Vert f-\frac1{m^j} \sum_{k\in D(M^j)} f(M^{-j}k) \mathscr{D}_{M^j}(\cdot-M^{-j}k)\bigg\Vert_p\le C \kappa_{j,p} \w E_{M^j}(f)_p,
\end{equation*}
where $f\in B$, $1\le p\le \infty$,
\begin{equation}\label{Kp}
  \kappa_{j,p}:=\left\{
          \begin{array}{ll}
            1, & \hbox{$1<p<\infty$,} \\
            j^d, & \hbox{$p=1,\,\infty$}
          \end{array}
        \right.
\end{equation}
and the constant $C$ does not depend on $f$ and $j$.
\end{example}

In the next example, we deal with a periodic Kantorovich-type quasi-interpolation operator gene\-rated by the samples $\{{\rm Avg}_{\s M^{-j}}f(M^{-j}k)\}_k$.

\begin{example}\label{e2}
Let $f\in L_p$, $1\le p\le \infty$, $\s\in (0,1]$, and $j\in \N$. Then
\begin{equation}\label{kd}
   \bigg\Vert f-\frac1{m^j} \sum_{k\in D(M^j)} {\rm Avg}_{\s M^{-j}} f(M^{-j}k) \mathscr{D}_{M^j,\s}^\chi(\cdot-M^{-j}k)\bigg\Vert_p\le C \kappa_{j,p} E_{M^j}(f)_p,
\end{equation}
where
\begin{equation*}
  \mathscr{D}_{M^j,\s}^\chi(x)=\sum_{\ell\in D(M^j)}\prod_{i=1}^d \frac{\pi \s m_i^{-j} \ell_i}{\sin \pi \s m_i^{-j} \ell_i}  {\rm e}^{2\pi {\rm i}(\ell,x)},
\end{equation*}
the constant $\kappa_{j,p}$ is given in~\eqref{Kp} and  $C$ does not depend on $f$ and $j$.
\end{example}

The proof of estimate~\eqref{kd} easily follows from inequality~\eqref{KS000Kant} with $\vp_j=\mathscr{D}_{M^j,\s}^\chi$
and $\w\vp_j=\s^{-d} m^j \chi_{M^{-j}[-\frac\s2,\frac\s2)^d}$. One only needs to take into account that~\eqref{sc} holds with $\rho=\d=1$,
$$
{\rm Avg}_{\s M^{-j}}f(x)=f*\w\vp_j (x)\sim \sum_{\ell\in \Z^d} \prod_{i=1}^d \frac{\sin \pi \s m_i^{-j} \ell_i}{\pi \s m_i^{-j} \ell_i}\h f(\ell){\rm e}^{2\pi {\rm i}(\ell,x)},
$$
$\sup_j \Vert \w\phi_j \Vert_{\mathcal{L}_{q,j}}<\infty$,
and $\sup_{\Vert f\Vert_p\le 1} \Vert f*\mathscr{D}_{M^j,\s}^\chi \Vert_p\le C\sup_{\Vert f\Vert_p\le 1} \Vert f*\mathscr{D}_{M^j} \Vert_p\le C\kappa_{j,p}$. The last estimate follows from the fact that the function
$$
\eta^\chi(\xi)=\eta(\xi) \prod_{i=1}^d \frac{\pi \s\xi_i}{\sin \pi \s\xi_i},
$$
where $\eta\in C^\infty (\R^d)$, $\eta(\xi)=1$ for $\xi\in[-1/2,1/2)^d$ and $\eta(x)=0$ for $\xi\not\in [-1,1)^d$, is a Fourier multiplier in $L_p(\R^d)$ for all $1\le p\le \infty$ (see Lemma~\ref{mult} below).

\subsection{Estimates of approximation in terms of moduli of smoothness and $K$-functionals}

We  need to introduce some additional notation.
For a given matrix $M$, $s\in \N$, and a function $f\in L_p$, we set
$$
\Omega_s(f,M^{-1})_p:=\sup_{|M\delta|<1,\delta\in \R^d} \Vert \Delta_\delta^s f\Vert_p,
$$
where
$$
\Delta_\delta^s f(x):=\sum_{\nu=0}^s (-1)^\nu \binom{s}{\nu} f(x+\delta \nu)
$$
and $\binom{\a}{\nu}=\frac{\a (\a-1)\dots (\a-\nu+1)}{\nu!}$, $\binom{\a}{0}=1$, for any $\a>0$.
This is the so-called (total) anisotropic modulus of smoothness.
Together with this modulus of smoothness, we will also use the classical mixed modulus of smoothness, which for a given vector $\beta\in \Z_+^d$ and a diagonal matrix $M={\rm diag}(m_1,\dots,m_d)$ is defined by
$$
\omega_\beta(f,M^{-1})_p:=\sup_{|\delta_i|<m_i^{-1},\,i=1,\dots,d} \Vert \Delta_{\delta_1 {\rm e}_1}^{\beta_1}\dots \Delta_{\delta_d {\rm e}_d}^{\beta_d} f\Vert_p.
$$

The following relations for the moduli of smoothness defined above were proved in~\cite{Tim}:
\begin{equation}\label{eqMOD1}
  \Omega_s(f,M^{-1})_{{p}}\asymp \sum_{i=1}^d \omega_{s {\rm e}_i}(f,M^{-1})_p,\quad f\in L_p,\quad 1<p<\infty,
\end{equation}
and
\begin{equation}\label{eqMOD2}
\Omega_s(f,M^{-1})_{{p}}\asymp \sum_{[\beta]=s,\,\beta\in \Z_+^d} \omega_{\beta}(f,M^{-1})_p,\quad f\in L_p,\quad 1\le p\le \infty,
\end{equation}
where $\asymp$ is a two-sided inequality with constants that do not depend on $f$ and $j$.

Let us recall several basic properties of moduli of smoothness (see, e.g.,~\cite[Ch.~4]{Nik}).
For  $f,g\in L_p$, $1\le p\le \infty$, and $s\in \N$, we have 
\begin{itemize}

\item[{\rm (a)}]
$              \Omega_s(f+g,M^{-1})_p\le
\Omega_s(f,M^{-1})_p+\Omega_s(g,M^{-1})_p;
$

 \item[{\rm (b)}]
$              \Omega_{s}(f,M^{-1})_p\le 2^s \Vert f\Vert_p$;

\item[{\rm (c)}]
for $\lambda>0$,
        $$\Omega_{s}(f,\lambda M^{-1})_p\le (1+\lambda)^s  \Omega_{s}(f,M^{-1})_p.$$

\smallskip

    \end{itemize}

We will also use the following  Jackson-type theorem in $L_p$ (see, e.g.,~\cite[Theorem 5.2.1 (7)]{Nik} or~\cite[5.3.2]{Timan}):

\begin{lemma}\label{lemJ}
  Let $f\in L_p$, $1\le p\le \infty$, and $s\in \N$. Then, there exists
$T_j\in \mathcal{T}_{M^j}$ such that
\begin{equation*}
  \Vert f-T_j\Vert_p\le C \sum_{i=1}^d \omega_{s {\rm e}_i}(f,M^{-j})_p,
\end{equation*}
where $C$ does not depend on $f$ and $T_j$.
\end{lemma}

The next lemma provides the Nikol'skii--Stechkin--Riesz type inequality  (see, e.g.~\cite[p.~215]{Timan}).
\begin{lemma}\label{lemNS}
Let $1\le p\le \infty$, $s\in\N$, and $n\in \N$. Then, for any trigonometric polynomial $T_n(x)=\sum_{|k|\le n} c_k {\rm e}^{2\pi {\rm i} kx}$,
$x\in\T$, we have
$$
\| T_n^{(s)}  \|_{L_p(\T)}\le
\left(\frac{n}{2\sin\frac{n\delta}{2}} \right)^s \| \Delta_\delta^s  T_n\|_{L_p(\T)},\quad 0<\delta\le 1/n.
$$
\end{lemma}

Recall that the sequence $\Lambda=\{\lambda_k\}_{k\in \Z^d}$ is called a Fourier multiplier in $L_p$, $1\le p\le \infty$, if for every function $f\in L_p$,
$$
\sum_{k\in \Z^d} \lambda_k \h  f(k){\rm e}^{2\pi {\rm i}(k,x)}
$$
is the Fourier series of a certain function $\Lambda f \in L_p$ and
$$
\Vert \{\lambda_k\}_{k} \Vert_{\mathcal{M}_p}=\sup_{\Vert f\Vert_p\le 1} \Vert \Lambda f \Vert_p.
$$

In the next theorem and  below, we denote $v_\d(\xi)=v(\d^{-1}\xi)$,
where $v\in C^\infty (\R^d)$, $v(\xi)=1$ for $\xi\in [-1/4,1/4)^d$ and $v(\xi)=0$ for $\xi\not\in [-3/8,3/8)^d$.

\begin{theorem}\label{corMOD1DD--}
Let $1\le p\le\infty$, $1/p+1/q=1$, $s\in \N$, $\delta\in (0,1/2)$,  and $j\in \N$. Suppose that $\w\phi_j\in \mathcal{D}'$ and $\phi_j \in \mathcal{T}_{M^j}$ are such that
\begin{equation}\label{zvezda1}
  \h{\vp_j}(k)\h{\w\vp_j}(k)=1+\sum_{[\beta]=s}(M^{-j}k)^\beta \Gamma_{j,s}(k)\quad \text{for all}\quad k\in D(\delta M^j),
\end{equation}
where
\begin{equation}\label{zvezda1m}
\sup_j \Vert \{\Gamma_{j,s}(k) v_\delta(M^{-j}k)\}_{k}\Vert_{\mathcal{M}_p}<\infty.
\end{equation}
Then, for any $f\in  B_{\w\vp_j, p}$, we have
\begin{equation}\label{KS000+NNNN--}
\begin{split}
\Vert f - Q_j(f,\phi_j,\w\phi_j) \Vert_p\le C\bigg(\Omega_s(f,M^{-j})_p+K_{\phi_j,q} \w E_{\frac \d 2 M^j}({{\w\phi_j}}*f)_p\bigg);
\end{split}
\end{equation}
if, additionally, $\w\phi_j\in L_q$, then for any $f\in L_p$, we have
\begin{equation}\label{KS000+NNNN--+}
  \Vert f - Q_j(f,\phi_j,\w\phi_j) \Vert_p\le C (1+K_{\phi_j,q}\Vert \w\phi_j \Vert_{\mathcal{L}_{q,j}})\Omega_s(f,M^{-j})_p,
\end{equation}
where the constant $C$ does not depend on $f$ and $j$.
\end{theorem}

\begin{proof}
To prove estimate~\eqref{KS000+NNNN--}, we will use the following slightly modified version of inequality~\eqref{KS000_c2}:
\begin{equation*}
       \Vert f - Q_j(f,\phi_j,\w\phi_j) \Vert_p\le C \Big(\Vert {\psi_j}*V_{\d M^j}f\Vert_{p}+E_{\frac \delta2 M^j}(f)_p+
     K_{\phi_j,q} \w E_{\frac\delta 2 M^j}({{\w\phi_j}}*f)_p\Big).
\end{equation*}
Thus, taking into account Lemma~\ref{lemJ} and relations~\eqref{eqMOD2}, we see that it is enough to show that
\begin{equation}\label{KKper0}
  \Vert {\psi_j}*V_{\d M^j}f\Vert_{p}\le C\Omega_s(f,M^{-j})_{{p}}.
\end{equation}

Using~\eqref{zvezda1}, \eqref{zvezda1m}, and Lemma~\ref{lemNS}, we derive
\begin{equation}\label{prr1}
  \begin{split}
     \Vert {\psi_j}*V_{\d M^j}f\Vert_{p}&\le \sum_{[\b]=s} \bigg\Vert \sum_{k} (M^{-j}k)^\b \Gamma_{j,s}(k) v_\d(M^{-j}k) \h f(k) {\rm e}^{2\pi \rm{i}(k,x)}\bigg\Vert_p\\
&\le C\sum_{[\b]=s} \bigg\Vert \sum_{k} (M^{-j}k)^\b  v(M^{-j}k) \h f(k) {\rm e}^{2\pi \rm{i}(k,x)}\bigg\Vert_p\\
&\le C\sum_{[\b]=s} \Big\Vert   \Delta_{\pi m_1^{-j}}^{\beta_1}\dots \Delta_{\pi m_d^{-j}}^{\beta_d}V_{M^j}f\Big\Vert_p\\
&\le C\Omega_s\(V_{M^j}f,M^{-j}\)_p.
  \end{split}
\end{equation}
Next, applying the properties of moduli of smoothness (a)--(c), inequality~\eqref{ve}, and Lemma~\ref{lemJ} along with~\eqref{eqMOD2}, we obtain
\begin{equation}\label{prr2}
  \begin{split}
      \Omega_s\(V_{M^j}f,M^{-j}\)_p &\le C\(2^s\Vert f-V_{M^j} f\Vert_p+\Omega_s(f,M^{-j})_p\)\\
      &\le C\Omega_s(f,M^{-j})_p.
   \end{split}
\end{equation}

Finally, combining~\eqref{prr1} and~\eqref{prr2}, we get~\eqref{KKper0}.

The proof of estimate~\eqref{KS000+NNNN--+} easily follows from Lemma~\ref{thKK}, Lemma~\ref{lemJ}, and inequality~\eqref{KKper0}.
\end{proof}


\subsubsection{Two-sided estimates of approximation and fractional smoothness}

Below, we will present some two-sided estimates of approximation by quasi-interpolation operators using fractional $K$-functionals and moduli of smoothness.

For our purposes, we will use the $K$-functional corresponding to the fractional Laplacian:
$$
\mathcal{K}_s^\Delta(f,M^{-1})_p:=\inf_g \{\Vert f-g\Vert_p+\Vert (-\Delta_{M^{-1}})^{s/2} g\Vert_p\},
$$
where
$$
(-\Delta_{M^{-1}})^{s/2} g(x)\sim \sum_{k\in \Z^d} |M^{-1} k|^s \h g(k){\rm e}^{2\pi{\rm i}(k,x)}.
$$

Recall that if $1<p<\infty$, $s>0$, and $M=\l I_d$, where $\l>1$ is integer, then the $K$-functional $\mathcal{K}_s^\Delta(f,M^{-1})_p$ is equivalent to the following fractional modulus of smoothness (see, e.g.,~\cite{Wil})
$$
\omega_s(f,\l^{-1})_p:=\sup_{|h|\le \l^{-1}} \bigg\Vert \sum_{l=0}^\infty (-1)^l \binom{s}{l}f(\cdot+h l)\bigg\Vert_p,
$$
i.e.,
\begin{equation}\label{KM}
  \mathcal{K}_s^\Delta(f,M^{-1})_p\asymp \omega_s(f,\l^{-1})_p,
\end{equation}
where $\asymp$ is a two-sided inequality with positive constants that do not depend on $f$ and $\l.$

\begin{theorem}\label{thfr1}
Let $1\le p\le\infty$, $1/p+1/q=1$, $s\in \N$, $\delta\in (0,1/2)$,  and $j\in \N$. Suppose that $\w\phi_j\in \mathcal{D}'$ and $\phi_j \in \mathcal{T}_{M^j}$ are such that
\begin{equation}\label{fr1}
   \sup_{j} \bigg\Vert \left\{   \frac{1-\h{\vp_j}(k)\h{\w \vp_j}(k)}{|M^{-j}k|^s}v_\d(M^{-j}k) \right\}_{k}\bigg\Vert_{\mathcal{M}_p}<\infty.
\end{equation}
Then, for any $f\in  B_{\w\vp_j, p}$, we have
\begin{equation}\label{fr2}
\begin{split}
\Vert f - Q_j(f,\phi_j,\w\phi_j) \Vert_p\le C\bigg(\mathcal{K}_s^\Delta(f,M^{-j})_p+K_{\phi_j,q} \w E_{\frac \d 2 M^j}({{\w\phi_j}}*f)_p\bigg);
\end{split}
\end{equation}
if, additionally, $\w\phi_j\in L_q$, then
\begin{equation}\label{fr2+}
\begin{split}
\Vert f - Q_j(f,\phi_j,\w\phi_j) \Vert_p\le C(1+K_{\phi_j,q}\Vert \w\phi_j \Vert_{\mathcal{L}_{q,j}})\mathcal{K}_s^\Delta(f,M^{-j})_p,
\end{split}
\end{equation}
where the constant $C$ does not depend on $f$ and $j$.
\end{theorem}

\begin{proof}[Proof] As in the proof of Theorems~\ref{corMOD1DD--}, it is sufficient to show that
\begin{equation}\label{fr4}
  \Vert {\psi_j}*V_{\d M^j}f\Vert_{p}\le C \mathcal{K}_s^\Delta(f,M^{-j})_p.
\end{equation}

Using condition~\eqref{fr1}, we derive
\begin{equation}\label{fr5}
  \begin{split}
     \Vert {\psi_j}*V_{\d M^j} f\Vert_{p}&=\bigg\Vert \sum_{k}    \frac{1-\h{\vp_j}(k)\h{\w \vp_j}(k)}{|M^{-j}k|^s} v_\d(M^{-j}k)v(M^{-j}k) |M^{-j}k|^s \h f(k) {\rm e}^{2\pi \rm{i}(k,x)}\bigg\Vert_p\\
&\le C \bigg\Vert \sum_{k} v(M^{-j}k) |M^{-j}k|^s \h f(k) {\rm e}^{2\pi \rm{i}(k,x)}\bigg\Vert_p\\
&=C\Vert (-\Delta_{M^{-j}})^{s/2}V_{M^j} f\Vert_p.
  \end{split}
\end{equation}
Next, taking into account the fact that
\begin{equation}\label{mmm}
  \sup_j\Vert \{v(M^{-j}k) |M^{-j}k|^s\}_{k} \Vert_{\mathcal{M}_p}<\infty\quad\text{for every}\quad s\ge 0
\end{equation}
(see Lemma~\ref{mult} below)
and choosing a function $g$ such that
$$
\Vert f-g\Vert_p+\Vert (-\Delta_{M^{-j}})^{s/2} g\Vert_p\le 2 \mathcal{K}_s^\Delta(f,M^{-j})_p,
$$
we obtain
\begin{equation}\label{fr6}
  \begin{split}
     \Vert (-\Delta_{M^{-j}})^{s/2}V_{M^j} f\Vert_p&\le \Vert (-\Delta_{M^{-j}})^{s/2}V_{M^j}(f-g)\Vert_p
     +\Vert (-\Delta_{M^{-j}})^{s/2}V_{M^j} g\Vert_p\\
&\le C\Vert f-g\Vert_p+\Big\Vert V_{M^j}\((-\Delta_{M^{-j}})^{s/2} g\)\Big\Vert_p\\
&\le C\(\Vert f-g\Vert_p+\Vert (-\Delta_{M^{-j}})^{s/2} g\Vert_p\)\le C\mathcal{K}_s^\Delta(f,M^{-j})_p.
  \end{split}
\end{equation}
Thus, combining~\eqref{fr5} and~\eqref{fr6}, we get~\eqref{fr4}. This implies that inequalities~\eqref{fr2} and~\eqref{fr2+} are valid.
\end{proof}


Now we consider the estimates from below.

\begin{theorem}\label{thfr1b}
Let $1\le p\le\infty$, $1/p+1/q=1$, $s>0$, $\delta\in (0,1/2)$,  and $j\in \N$. Suppose that $\w\phi_j\in \mathcal{D}'$ and $\phi_j \in \mathcal{T}_{M^j}$ are such that
\begin{equation}\label{fr1b}
   \sup_{j} \bigg\Vert \left\{   \frac{|M^{-j}k|^s}{1-\h{\vp_j}(k)\h{\w \vp_j}(k)}v_{1/\d}(M^{-j}k) \right\}_{k}\bigg\Vert_{\mathcal{M}_p}<\infty.
\end{equation}
Then, for any $f\in  B_{\w\vp_j, p}$, we have
\begin{equation}\label{fr2b}
\begin{split}
\mathcal{K}_s^\Delta(f,M^{-j})_p\le C\(\Vert f - Q_j(f,\phi_j,\w\phi_j) \Vert_p+E_{\frac12 M^j}(f)_p+K_{\phi_j,q} \w E_{\frac12M^j}({{\w\phi_j}}*f)_p\);
\end{split}
\end{equation}
if, additionally, $\w\phi_j\in L_q$, then for any $f\in L_p$, we have
\begin{equation}\label{fr3b}
\begin{split}
\mathcal{K}_s^\Delta(f,M^{-j})_p\le C(1+K_{\phi_j,q}\Vert \w\phi_j \Vert_{\mathcal{L}_{q,j}})\Vert f - Q_j(f,\phi_j,\w\phi_j) \Vert_p,
\end{split}
\end{equation}
where the constant $C$ does not depend on $f$ and $j$.
\end{theorem}

\begin{remark}\label{rem1-}
If in Theorem~\ref{thfr1b} instead of~\eqref{fr1b}, we suppose that
\begin{equation*}
   \sup_{j} \bigg\Vert \left\{   \frac{|M^{-j}k|^s}{1-\h{\vp_j}(k)\h{\w \vp_j}(k)}\chi_{D(M^j)}(k)\right\}_{k}\bigg\Vert_{\mathcal{M}_p}<\infty,
\end{equation*}
then, for any $f\in  B_{\w\vp_j, p}$, $1<p<\infty$, we have
\begin{equation*}
\begin{split}
\mathcal{K}_s^\Delta(f,M^{-j})_p\le C\(\Vert f - Q_j(f,\phi_j,\w\phi_j) \Vert_p+K_{\phi_j,q} \w E_{M^j}({{\w\phi_j}}*f)_p\).
\end{split}
\end{equation*}
This follows from the proof of Theorem~\ref{thfr1b} presented below and Corollary~\ref{cor1thKS} a).
\end{remark}

\begin{remark}\label{rem2}
If  $d=1$ and in conditions~\eqref{fr1} or~\eqref{fr1b} we replace $|M^{-j}k|^s$ with $({\rm i}M^{-j}k)^s$, $M>1$, then for any $f\in L_p$, $1\le p\le \infty$, and $s>0$,
the $K$-functional  $\mathcal{K}_s^\Delta(f,M^{-j})_p$ can be replaced with the fractional modulus of smoothness ${\omega}_s(f,M^{-j})_p$. This easily follows from the proofs of Theorems~\ref{thfr1} and~\ref{thfr1b} and the fact that for any $f\in L_p(\T)$ and $s>0$ (see, e.g.,~\cite{BDGS77})
$$
{\omega}_s(f,t)_p\asymp \inf_g \(\Vert f-g\Vert_p+t^s \Vert g^{(s)}\Vert_p\),
$$
where $\asymp$ is a two-sided inequality with positive constants that do not depend on $f$ and $t$.
\end{remark}

\begin{proof}[Proof of Theorem~\ref{thfr1b}]
By the definition of the $K$-functional, we derive
\begin{equation}\label{fr4b}
  \begin{split}
     \mathcal{K}_s^\Delta(f,M^{-j})_p\le \Vert f-Q_j(f,\vp_j,\w\vp_j)\Vert_p+\Vert (-\Delta_{M^{-j}})^{s/2} Q_j(f,\vp_j,\w\vp_j)\Vert_p.
  \end{split}
\end{equation}
Let $T_j\in \mathcal{T}_{M^j}$ be some trigonometric polynomial that will be chosen later. Taking into account condition~\eqref{fr1b} and using~\eqref{mmm} and equality~\eqref{zvezda1ra}, we obtain
\begin{equation}\label{fr5b}
  \begin{split}
    \Vert (-&\Delta_{M^{-j}})^{s/2} Q_j(f,\vp_j,\w\vp_j)\Vert_p\\
&\le \Vert (-\Delta_{M^{-j}})^{s/2} \(Q_j(f,\vp_j,\w\vp_j)-T_j\)\Vert_p+\Vert (-\Delta_{M^{-j}})^{s/2} T_j\Vert_p\\
&\le C \(\Vert Q_j(f,\vp_j,\w\vp_j)-T_j\Vert_p+\Vert {\psi_j}*T_j\Vert_p\)\\
&= C \(\Vert Q_j(f,\vp_j,\w\vp_j)-T_j\Vert_p+\Vert Q_j(T_j,\vp_j,\w\vp_j)-T_j\Vert_p\)\\
&\le C \(\Vert f-Q_j(f,\vp_j,\w\vp_j)\Vert_p+\Vert f-T_j\Vert_p+\Vert Q_j(f-T_j,\vp_j,\w\vp_j)\Vert_p\).\\
  \end{split}
\end{equation}

Now, to prove inequality~\eqref{fr2b}, we choose $T_j=V_{M^j} f$. Then, applying estimates~\eqref{I3} and~\eqref{zvezzvez}, we derive
\begin{equation}\label{fr6b}
  \begin{split}
    \Vert Q_j(f-T_j,\vp_j,\w\vp_j)\Vert_p&\le  C K_{\phi_j,q} \(\w E_{M^j}({{\w\phi_j}}*f)_p+\Vert {{\w\phi_j}}*(f -T_j)\Vert_p\)\\
&\le  C K_{\phi_j,q} \(\w E_{M^j}({{\w\phi_j}}*f)_p+\w E_{\frac12M^j}({{\w\phi_j}}*f)_p\)\\
&\le C K_{\phi_j,q} \w E_{\frac12M^j}({{\w\phi_j}}*f)_p.
  \end{split}
\end{equation}
Using also estimate~\eqref{ve}, we see that inequalities~\eqref{fr6b} and~\eqref{fr5b} imply that
\begin{equation}\label{fr7b}
\begin{split}
     \Vert (-&\Delta_{M^{-j}})^{s/2} Q_j(f,\vp_j,\w\vp_j)\Vert_p\\
&\le C\(\Vert f-Q_j(f,\vp_j,\w\vp_j)\Vert_p+E_{\frac12 M^j}(f)_p+K_{\phi_j,q} \w E_{\frac12M^j}({{\w\phi_j}}*f)_p\).
\end{split}
\end{equation}
Combining~\eqref{fr4b} and~\eqref{fr7b}, we get~\eqref{fr2b}.

To prove inequality~\eqref{fr3b}, it is enough to set $T_j=Q_j(f,\vp_j,\w\vp_j)$ and take into account that by~\eqref{fr5b} and~\eqref{zvezzvez2}, we have
\begin{equation*}
  \Vert (-\Delta_{M^{-j}})^{s/2} Q_j(f,\vp_j,\w\vp_j)\Vert_p\le C(1+K_{\phi_j,q}\Vert \w\phi_j \Vert_{\mathcal{L}_{q,j}})\Vert f - Q_j(f,\phi_j,\w\phi_j) \Vert_p,
\end{equation*}
which together with~\eqref{fr4b}  implies~\eqref{fr3b}.
\end{proof}


In the next results, we deal with functions/distributions $\phi_j$ and $\w\phi_j$ having the following special form:
\begin{equation}\label{pipi}
  \vp_j(x)\sim\sum_{k\in \Z^d} \Phi(M^{-j}k){\rm e}^{2\pi {\rm i}(k,x)}, \quad
\w\vp_j(x)\sim\sum_{k\in \Z^d} \w\Phi(M^{-j}k){\rm e}^{2\pi {\rm i}(k,x)},
\end{equation}
where $\Phi, \,\w\Phi\,:\, \R^d\to \mathbb{C}$ are appropriate functions, which will be specified below.
Actually, most of the quasi-interpolation operators~\eqref{oper} are defined by means of functions/distributions $\vp_j$ and  $\w\vp_j$ given by~\eqref{pipi}. Below, we would like to give a version of Theorem~\ref{corMOD1DD--}, in which the conditions on $\vp_j$ and  $\w\vp_j$ are given only in terms of some simple smoothness properties of the functions $\Phi$ and $\w\Phi$.

For our purposes, we need to recall some facts about Fourier multipliers on $L_p(\R^d)$. First, we recall that a bounded function $\mu \,:\,\R^d\to \C$ is called a Fourier multiplier on $L_p(\R^d)$, $1\le p\le \infty$ (we will write $\mu \in \mathcal{M}_p(\R^d)$), if the operator $T_\mu$ defined by 
$$
\mathcal{F}(T_\mu f)=\mu \mathcal{F}(f),\quad f\in L_p(\R^d)\cap L_2(\R^d),
$$
is bounded on $L_p(\R^d)$, i.e., there exists a constant $C$ such that $\|T_\mu f\|_{L_p(\R^d)}\le C\|f\|_{L_p(\R^d)}$. The norm of the Fourier multiplier $\mu$ is given by
$$
\Vert \mu\Vert_{\mathcal{M}_p(\R^d)}=\sup_{\Vert f\Vert_{L_p(\R^d)}\le 1} \Vert T_\mu f\Vert_{L_p(\R^d)}.
$$

We will use the following basic properties of Fourier multipliers on $L_p(\R^d)$:

\begin{lemma}\label{mult}
a) If $\mu \in \mathcal{M}_p(\R^d)$, $1\le p\le \infty$, and $\mu(t)$ is continuous at the points $t\in \Z^d$, then, for any dilation matrix $M$ and $j\in \N$, the sequence $\{\mu(M^{-j}k)\}_{k\in \Z^d}$ is a bounded Fourier multiplier in the space $L_p(\T^d)$ and
$$
\sup_j \Vert \{\mu(M^{-j}k)\}_{k}\Vert_{\mathcal{M}_p}\le c(p,d) \Vert \mu\Vert_{\mathcal{M}_p(\R^d)}.
$$

b) Suppose that the function $\mu$ belongs to $C(\R^d)$ and has a compact support. If $\mu \in W_s^d(\R^d)$ for some $s>1$, or more generally $\mathcal{F}(\mu) \in L_1(\R^d)$, then $\mu \in \mathcal{M}_p(\R^d)$ for all $1\le p\le \infty$.
\end{lemma}

\begin{proof}
a) This assertion follows from the well-known de Leeuw theorem (see~\cite{deL}) and the fact that for every affine transformation $l\,:\, \R^d\to \R^d$, we have $\Vert \mu\circ l\Vert_{\mathcal{M}_p(\R^d)}=\Vert \mu\Vert_{\mathcal{M}_p(\R^d)}$ (see, e.g.,~\cite[p.~147]{Gr}).

b)  The assertion can be found, e.g., in~\cite{LST}.
\end{proof}

\begin{remark}\label{rem1}
The sufficient condition for Fourier multipliers given in assertion b)  is one of the simplest and is rather rough.
For more advanced sufficient conditions for Fourier multipliers see, e.g.,~\cite[Ch. 5]{Gr}, \cite{LST}, \cite{K14}.
\end{remark}

Now, we are ready to present an analogue of Theorem~\ref{corMOD1DD--}.

\begin{theorem}\label{corMOD1DD++}
Let $1\le p\le\infty$, $1/p+1/q=1$, $s\in \N$, $\d\in (0,1/2)$, and $j\in \N$. Suppose that $\w\phi_j\in \mathcal{D}'$ and
$\phi_j \in \mathcal{T}_{M^j}$, $\phi_j$ and $\w\phi_j$ are given by~\eqref{pipi},
$\Phi, \w\Phi\in C^{s+d}(2\d \T^d)$ and $D^\alpha (1-{\w\Phi}\Phi)(0)=0$ for all $|\alpha|<s$.
Then, for any $f\in  B_{\w\vp_j, p}$, we have
\begin{equation*}
\begin{split}
\Vert f - Q_j(f,\phi_j,\w\phi_j) \Vert_p\le C\bigg(\Omega_s(f,M^{-j})_p+K_{\phi_j,q} \w E_{\frac\delta2 M^j}({{\w\phi_j}}*f)_p\bigg),
\end{split}
\end{equation*}
if, additionally, $\w\phi_j\in L_q$, then for any $f\in L_p$, we have
\begin{equation*}
  \Vert f - Q_j(f,\phi_j,\w\phi_j) \Vert_p\le C(1+K_{\phi_j,q}\Vert \w\phi_j \Vert_{\mathcal{L}_{q,j}})\Omega_s(f,M^{-j})_p,
\end{equation*}
where the constant $C$ does not depend on $f$ and $j$.
\end{theorem}

\begin{proof}
The proof easily follows from Theorem~\ref{corMOD1DD--} and Lemma~\ref{mult}. One only needs  to take into account that
using Taylor's formula near zero,  we have
$$
\Phi(\xi)\w\Phi(\xi)=1+\sum_{[\beta]=s}\frac{s}{\beta!} r^\beta \int_0^1 (1-t)^{s-1}D^\beta {\Phi}  {\w \Phi} (t\xi) {\rm d}t,\quad \beta \in \Z_+^d,\ [\beta]=s.
$$
Then,  denoting
$$
G_\beta(\xi)=\rho(\xi)\int_0^1 (1-t)^{s-1}D^\beta {\Phi}  {\w \Phi} (t\xi) {\rm d}t,
$$
where $\rho(\xi)\in C^\infty (\R^d)$, $\rho(\xi)=1$ for $\xi \in \d \T^d$ and $\rho(\xi)=0$ for $\xi\not\in 2\d \T^d$, and taking into account that $G_\beta \in C^d(\R^d)$, we have that by Lemma~\ref{mult}, conditions~\eqref{zvezda1} and~\eqref{zvezda1m} hold with
$\Gamma_{j,\beta}(k)=G_\beta(M^{-j}k)$.
\end{proof}

\begin{example}\label{e3}
Taking $\w\vp_j=m^j \chi_{M^{-j} [-\frac12,\frac12)^d}$ and $\vp_j=\mathscr{D}_{M^j}$, it is not difficult to see that   Theorem~\ref{corMOD1DD++} provides the following error estimate for the corresponding Kantorovich-type operator (cf.~\cite[Proposition~19]{KS3}):
\begin{equation}\label{e31}
 \bigg\Vert f- \frac1{m^j}\sum_{k\in D(M^j)}{\rm Avg}_{\s M^{-j}} f(M^{-j}k) \mathscr{D}_{M^j}(\cdot-M^{-j}k)\bigg\Vert_p\le C \kappa_{j,p}\Omega_2(f,M^{-j})_p,
\end{equation}
where $f\in L_p$, $1\le p\le \infty$, $\s\in (0,1]$, the constant $\kappa_{j,p}$ is given in~\eqref{Kp}, and  $C$ does not depend on $f$ and $j$.
\end{example}

We omit the formulations of the corresponding analogues of Theorems~\ref{thfr1} and~\ref{thfr1b} in terms of the smoothness properties of $\Phi$ and $\w\Phi$. Using Lemma~\ref{mult} and Remark~\ref{rem1}, one can directly and easily obtain appropriate statements.
Instead of this, we give several examples of applications of Theorems~\ref{thfr1} and~\ref{thfr1b} for some special quasi-interpolation operators.

First, we consider an estimate from below for the $L_p$-error of approximation by the quasi-interpolation operator from Example~\ref{e3}.

\begin{example}\label{e4}
Using Remark~\ref{rem1-} and Lemma~\ref{mult}, we obtain that for any
$f\in L_p$, $1<p<\infty$, $\s\in (0,1]$, and $j\in \N$
\begin{equation*}
  C\mathcal{K}_2^\Delta(f,M^{-j})_p\le \bigg\Vert f- \frac1{m^j}\sum_{k\in D(M^j)} {\rm Avg}_{\s M^{-j}}f(M^{-j}k) \mathscr{D}_{M^j}(\cdot-M^{-j}k)\bigg\Vert_p,
\end{equation*}
where $C$ does not depend on $f$ and $j$. Combining this estimate and inequality~\eqref{e31}, we derive that
\begin{equation*}
  \bigg\Vert f- \frac1{m^j}\sum_{k\in D(M^j)} {\rm Avg}_{\s M^{-j}} f(M^{-j}k) \mathscr{D}_{M^j}(\cdot-M^{-j}k)\bigg\Vert_p\asymp \Omega_2(f,M^{-j})_p.
\end{equation*}
In the last estimate, we took into account the fact that $\Omega_2(f,M^{-j})_p\le C\mathcal{K}_2^\Delta(f,M^{-j})_p$, which easily follows from
relation~\eqref{eqMOD1} and inequality $\Vert \Delta_h^2 g\Vert_{L_p(\T)}\le \Vert g'' \Vert_{L_p(\T)}$.
\end{example}

Our next example concerns quasi-projection operators that are generated by an average sampling instead of the exact samples of~$f$. Note that in the non-periodic case such operators are useful to reduce noise (see, e.g.,~\cite{ZWS}).
However, we will show that some of these operators cannot provide as "good" an approximation order as in the case of the classical interpolation operator, cf. Example~\ref{e1}.

\begin{example}\label{e5}
Let $d=1$ and $M\in \N$, $M\ge 2$. For $f\in B$,  we denote
$$
\l_j f(x)=\frac14 f(x-M^{-j-1})+\frac12 f(x)+\frac14 f(x+M^{-j-1})\sim \sum_{\ell\in \Z} \h{\w\vp_j}(\ell)\h f(\ell) {\rm e}^{2\pi {\rm i}\ell x},
$$
where $\h{\w\vp_j}(\ell)=\cos^2 (2\pi M^{-j-1}\ell)$. Using Theorems~\ref{thfr1} and~\ref{thfr1b} and Lemma~\ref{mult} for $\w\vp_j$ and $\vp_j=\mathcal{D}_{M^j}$, taking also into account Remark~\ref{rem2}, we derive
\begin{equation*}
\begin{split}
     C_1\omega_2(f,M^{-j})_p\le \bigg\Vert f- \frac1{M^j}\sum_{k\in D(M^j)} \l_j  &f(M^{-j}k) \mathscr{D}_{M^j}(\cdot-M^{-j}k)\bigg\Vert_p\\
&\le C_2\(\omega_2(f,M^{-j})_p+\w E_{M^j}(\l_j f)_p\),
\end{split}
\end{equation*}
where $1<p<\infty$ and $C_1$, $C_2$ are some positive constants that do not depend on $f$ and $j$.
\end{example}

Finally, we present two examples of the error estimates, in which we essentially use  the fractional smoothness of a function $f$.
For our purposes, we consider the following Riesz kernel
$$
\mathscr{R}_{s,M^j}^\gamma(x)=\sum_k (1-|c_d M^{-j}k|^s)_+^\g {\rm e}^{2\pi {\rm i}(k,x)},\quad s,\g>0\quad\text{and}\quad c_d=4 d^{1/2}.
$$

\begin{example}\label{e6}
Let $1\le p\le \infty$, $s>0$, $\g>\frac{d-1}{2}$, and $j\in \N$.
  \begin{itemize}
    \item[1)] For any $f\in  B$ ($f\in C(\T^d)$ in the case $p=\infty$), we have
\begin{equation}\label{ee61}
\begin{split}
C_1\mathcal{K}_s^\Delta(f,M^{-j})_p\le \bigg\Vert f- \frac1{m^j}\sum_{k\in D(M^j)} &f(M^{-j}k) \mathscr{R}_{s,M^j}^\gamma (\cdot-M^{-j}k)\bigg\Vert_p\\
&\le C_2\bigg(\mathcal{K}_s^\Delta(f,M^{-j})_p+\w E_{c M^j}(f)_p\bigg),
\end{split}
\end{equation}
where $c$, $C_1$ and $C_2$ are some positive constants that do not depend on $f$ and $j$
    \item[2)] For any $f\in  L_p$, $s\in (0,2]$, and $\s\in (0,1]$, we have
\begin{equation}\label{ee62}
\begin{split}
\bigg\Vert f- \frac1{m^j}\sum_{k\in D(M^j)} {\rm Avg}_{\s M^{-j}} f(M^{-j}k) &\mathscr{R}_{s,M^j}^\gamma (\cdot-M^{-j}k)\bigg\Vert_p\asymp\mathcal{K}_s^\Delta(f,M^{-j})_p,
\end{split}
\end{equation}
where $\asymp$ is a two-sided inequality with positive constants that do not depend on $f$ and $j$.
  \end{itemize}
\end{example}

The proof of inequalities in~\eqref{ee61} follows from Theorems~\ref{thfr1} and~\ref{thfr1b}, Lemma~\ref{mult}, and the fact that with an appropriate parameter $\d \in (0,1/2)$, the Fourier transforms of the functions
$$
g_1(\xi)=\frac{|\xi|^s v_{1/\d}(\xi)}{1-(1-|c_d \xi|^s)_+^\g}\quad\text{and}\quad g_2(\xi)=\frac{1-(1-|c_d \xi|^s)_+^\g v_{\d}(\xi)}{|\xi|^s}
$$
belong to $L_1(\R^d)$ (see, e.g.,~\cite{R10}, see also the proof of Theorem~2 in~\cite{K12}).

The proof of~\eqref{ee62} is similar. In this case, one only needs to investigate, by analogy with the previous case, the following two functions
$$
g_2(\xi)=\frac{|\xi|^s v_{1/\d}(\xi)}{1-\w\Phi(\xi)(1-|c_d \xi|^s)_+^\g}\quad\text{and}\quad g_3(\xi)=\frac{1-\w\Phi(\xi)(1-|c_d \xi|^s)_+^\g v_{\d}(\xi)}{|\xi|^s},
$$
where $\w\Phi(\xi)= \prod_{\ell=1}^d \frac{\sin \pi \s \xi_\ell}{\pi\s \xi_\ell}$.

\subsection{Error estimates for functions from Besov-type spaces}

In the previous sections, we obtained error estimates for the quasi-interpolation operators $Q_j(f,\vp_j,\w\vp_j)$ under very general conditions on the distribution $\w\vp_j$. These estimates were given in terms of the best one-sided approximation $\w E_{\d M^j}({{\w\phi_j}}*f)_p$ and appropriate moduli of smoothness and $K$-functionals.
At the same time, we proved that in the case $\w\vp_j\in L_q$, the best one-sided approximation can be replaced by the classical best approximation $E_{\d M^j}(f)_p$. In this section, we will present other possibilities (not so restrictive as the assumption $\w\vp_j\in L_q$) to avoid exploitation of a quite specific quantity $\w E_{\d M^j}({{\w\phi_j}}*f)_p$.

First of all, we note that the best one-sided approximation can be estimated from above by means of the so-called $\tau$-modulus of smoothness, which is defined by
$$
\tau_s(g,u)_p:=\Vert \omega (g,\cdot,u) \Vert_p,\quad s\in \N,\quad u>0,
$$
where
$$
\omega (g,x,u)=\sup\{|\Delta_h^s g(t)|\,:\, t,t+sh\in D(su,x)\}, \quad x\in \R^d,
$$
$$
D(u,x)=\{y\in \R^d\,:\, |x-y|\le u/2\}.
$$
Recall (see~\cite{AP}) that for any $g\in B$, $s\in \N$, and the isotropic matrix $M=\l I_d$, $\l>1$ we have
\begin{equation}\label{tau1}
  \w E_{M^{j}}(g)_p\le C_{s,d} \tau_s (g,\l^{-j})_p,
\end{equation}
where the constant $C$ does not depend on $g$ and $j$.

For smooth functions, one can estimate one-sided best approximation as follows (see~\cite{Popov}):
if $f\in W_p^d \cap B$, then
\begin{equation}\label{tau2}
\w E_{M^{j}}(g)_p \le C_d\sum_{\a_j\in \{0,1\},\, [\a]>0}\l^{-j[\a]} E_{M^j}(D^\a g)_p.
\end{equation}

Thus, using~\eqref{tau1} or~\eqref{tau2} with $g={{\w\phi_j}}*f$, we can replace  $\w E_{\d M^j}({{\w\phi_j}}*f)_p$ in Theorems~\ref{cor1}--\ref{corMOD1DD++} by the corresponding approximation quantity from the right-hand sides of~\eqref{tau1} or~\eqref{tau2}.

Below, using a special Besov space, we present another approach to replace  $\w E_{\d M^j}({{\w\phi_j}}*f)_p$ in the corresponding results. Note that this approach is based on some ideas from~\cite{H} and~\cite{KKS2}. In contrast to formulas~\eqref{tau1} and~\eqref{tau2}, we avoid calculations of special $\tau$-moduli of smoothness and the consideration of functions from the Sobolev spaces.

We use the following anisotropic Besov spaces with respect to the matrix~$M$. We say that
$f\in \mathbb{B}_{p,q}^s (M)$, $1\le p\le\infty$, $0<q\le \infty$, and $s>0$, if $f\in L_p$ and
$$
\Vert f\Vert_{\mathbb{B}_{p,q}^s (M)}:=\Vert f\Vert_p+\(\sum_{\nu=1}^\infty m^{\frac sd q\nu} E_{M^\nu} (f)_p^q\)^{\frac 1q}<\infty.
$$

For our purposes, we need to specify the class of tempered distributions $\w\phi_j$. We say that a sequence of tempered distribution $\w\phi_j$ belongs to the class $\mathcal{D}_{N,j,p}'$ for some $N\ge 0$ and $1\le p\le \infty$ if there exists a positive constant $C$, which does not depend on $j$,  such that for any trigonometric polynomial $T_\nu \in \mathcal{T}_{M^\nu}$, one has
\begin{equation}\label{DefS}
  \Vert {\w\phi_j}*T_\nu \Vert_p \le C m^{\frac Nd (\nu-j)} \Vert T_\nu \Vert_p\quad\text{for all}\quad \nu\ge j,\quad j,\nu\in \N.
\end{equation}

As a simple example of $\w\phi_j \in \mathcal{D}_{N,j,p}'$, we can take the distribution corresponding to some differential operator. Namely, if we set
$$
\h{\w\phi_j} (\ell) = \sum_{[\beta]\le N} c_\beta (2\pi {\rm i} M^{-j} \ell)^{\beta}, \quad N\in \Z_+,
$$
where the numbers $c_\beta$ do not depend on $j$,  then by the well-known Bernstein inequality for trigonometric polynomials (see, e.g.,~\cite[p.~215]{Timan})
$$
\bigg\Vert \sum_{k=-n}^n ({\rm i}k)^r a_k {\rm e}^{2\pi {\rm i}kx} \bigg\Vert_{L_p(\T)} \le n^r \bigg\Vert \sum_{k=-n}^n a_k {\rm e}^{2\pi {\rm i}kx} \bigg\Vert_{L_p(\T)},
$$
we can easily derive that $\w\phi_j\in \mathcal{D}_{N,j,p}'$.

\begin{lemma}\label{bes1}
Let $1\le p\le\infty$, $M\ge 0$, $\d\in (0,1]$, $j\in \N$, and $\w\phi_j \in \mathcal{D}_{N,j,p}'$. Then, for any
$f\in \mathbb{B}_{p,1}^{N+d/p}(M)$,
\begin{equation}\label{bes2}
  \sum_{\ell\in\zd}
 {\h{\w\phi_j}(\ell)} \h f(\ell)    {\rm e}^{2 \pi {\rm i} (\ell,x)}
\end{equation}
is a Fourier series of a continuous function  ${{\w \vp_j}}*f$ on $\T^d$, i.e.,  $\mathbb{B}_{p,1}^{N+d/p}(M)\subset B_{\w\vp_j,p}$, and
\begin{equation}\label{bes3}
  \Vert \{{{\w \vp_j}}*f(M^{-j}k)-{{\w \vp_j}}*T_j(M^{-j}k)\}_{k}\Vert_{\ell_{p,M^j}}\le C m^{-(\frac1p+\frac Nd)j}\sum_{\nu=j}^\infty m^{(\frac 1p+\frac Nd)\nu} E_{\d M^\nu} (f)_p,
\end{equation}
where $T_j\in \mathcal{T}_{M^j}$ is such that $\Vert f-T_j\Vert\le c(d,p,\d) E_{\d M^j}(f)_p$ and the constant $C$ does not depend on $f$ and~$j$.
\end{lemma}

\begin{proof}
First, we show that the series in~\eqref{bes2} is a Fourier series of a certain continuous function, which we will denote by ${{\w \vp_j}}*f$.

Using Nikolskii's inequality of different metrics (see, e.g.,~\cite[p.~133]{Nik})
$$
\Vert T_\nu\Vert_\infty\le C_p m^\frac\nu p \Vert T_\nu\Vert_p
$$
and inequality~\eqref{DefS}, we derive
\begin{equation}\label{bes4}
  \begin{split}
     \sum_{\nu=1}^\infty \Vert {{\w \vp_j}}*T_{\nu+1}-{{\w \vp_j}}*T_{\nu} \Vert_\infty&\le C\sum_{\nu=1}^\infty m^\frac\nu p \Vert {{\w \vp_j}}*(T_{\nu+1}-T_{\nu}) \Vert_p\\
&\le C m^{-\frac Nd j} \sum_{\nu=1}^\infty m^{(\frac1p+\frac Nd)\nu} \Vert T_{\nu+1}-T_{\nu} \Vert_p\\
&\le C m^{-\frac Nd j} \sum_{\nu=1}^\infty m^{(\frac1p+\frac Nd)\nu} E_{\d M^\nu}(f)_p.
  \end{split}
\end{equation}
The estimates~\eqref{bes4} imply that the sequence $\{{\w \vp_j}*T_{\nu}\}_{\nu \in \N}$ is fundamental in $C(\T^d)$. We denote its limit  by ${\w \vp_j}*f$. It is clear that this limit does not depend on the choice of polynomials $T_\nu$. Thus, if $T_\nu$ is defined using the de la Vall\'ee Poussin means $V_\nu f$, we derive that $\{ \h{\w\vp_j}(\ell) \h f(\ell)  \}_{\ell}$ are the Fourier coefficients of the function  ${\w \vp_j}*f$ since for a fixed $\ell$ and a sufficiently large $\nu$
\begin{equation*}
\begin{split}
    |\h{{\w \vp_j}*f}(\ell)-\h{\w\vp_j}(\ell) \h f(\ell)|&=\bigg|\int_{\T^d} ({\w \vp_j}*f(x)-{\w \vp_j}*V_\nu f(x)) {\rm e}^{2\pi {\rm i}(\ell,x)}{\rm d}x\bigg|\\
&\le \Vert {\w \vp_j}*f-{\w \vp_j}*V_\nu f\Vert_\infty\to 0\quad\text{as}\quad \nu\to \infty.
\end{split}
\end{equation*}

Now, we prove inequality~\eqref{bes3}. Using the representation
$$
{\w \vp_j}*f -{\w \vp_j}*T_j=\sum_{\nu=j}^\infty {\w \vp_j}*(T_{\nu+1}-T_\nu)\quad\text{in}\quad C(\T^d),
$$
Lemma~\ref{lemMZ}, and~\eqref{bes4}, we obtain
\begin{equation*}
  \begin{split}
     \Vert \{{{\w \vp_j}}*f(M^{-j}k)-&{{\w \vp_j}}*T_j(M^{-j}k)\}_{k}\Vert_{\ell_{p,M^j}}\le \sum_{\nu=j}^\infty \Vert \{{{\w \vp_j}}*(T_{\nu+1}-T_\nu)(M^{-j}k)\}_{k}\Vert_{\ell_{p,M^j}}\\
&\le m^{-\frac jp}\sum_{\nu=j}^\infty m^{\frac \nu p}\Vert \{{{\w \vp_j}}*(T_{\nu+1}-T_\nu)(M^{-\nu}k)\}_{k}\Vert_{\ell_{p,M^\nu}}\\
&\le Cm^{-\frac jp}\sum_{\nu=j}^\infty m^{\frac \nu p}\Vert {{\w \vp_j}}*(T_{\nu+1}-T_\nu)\Vert_{p}\\
&\le Cm^{-\frac jp}\sum_{\nu=j}^\infty m^{\frac Nd(\nu+1-j)\frac \nu p}\Vert T_{\nu+1}-T_\nu\Vert_{p}\\
&\le Cm^{-(\frac1p+\frac Nd)j}\sum_{\nu=j}^\infty m^{(\frac1p+\frac Nd)}E_{\d M^\nu}(f)_{p},
  \end{split}
\end{equation*}
which proves the lemma.
\end{proof}

We have the following counterpart of Lemma~\ref{thKS}:

\begin{lemma}\label{thKS+B}
Let $1\le p\le\infty$, $1/p+1/q=1$, $\d\in (0,1]$,  and $j\in \N$. Suppose that $\w\phi_j\in \mathcal{D}_{N,j,p}'$ and
$\phi_j \in \mathcal{T}_{M^j}$.
Then, for any $f\in \mathbb{B}_{p,1}^{d/p+N}(M)$, we have
\begin{equation}\label{KS000B}
\begin{split}
     \Vert f - Q_j(f,\phi,\w\phi) \Vert_p\le C\(\Vert {\psi_j}*T_j\Vert_{p}+m^{-j(\frac1p+\frac Nd)}\sum_{\nu=j}^\infty m^{(\frac1p+\frac Nd)\nu} E_{\delta M^\nu}(f)_p\),
\end{split}
\end{equation}
where $\psi_j$ is given in \eqref{psi1}, $T_j\in \mathcal{T}_{M^j}$ is such that $\Vert f-T_j\Vert_p\le c(d,p,\d)E_{\d M^j}(f)_p$,
and the constant $C$ does not depend on $f$ and $j$.
\end{lemma}

\begin{proof}
The proof is similar to the one of Lemma~\ref{thKS}. The only difference consists in the estimate of the norm $I_3$ in inequality~\eqref{I3}.
In particular, using Lemma~\ref{bes1} and the first inequality in~\eqref{I3}, we derive that
\begin{equation}\label{I3n}
  \begin{split}
           I_3&\le C K_{\phi_j,q} \(\frac1{m^j} \sum_{k\in D(M^j)}|{{\w\phi_j}}*f(M^{-j} k)-{{\w\phi_j}}*T_j(M^{-j} k)|^p\)^\frac1p\\
&\le C K_{\phi_j,q} m^{-(\frac1p+\frac Nd)j}\sum_{\nu=1}^\infty m^{(\frac 1p+\frac Nd)\nu} E_{\d M^\nu} (f)_p.
   \end{split}
\end{equation}
Thus, combining~\eqref{I123}, \eqref{I2}, and~\eqref{I3n}, we prove the lemma.
\end{proof}

\begin{remark}
If in Lemma~\ref{thKS+B} we replace the condition $\w\phi_j \in \mathcal{D}_{0,j,\infty}'$ by
\begin{equation}\label{condinf}
  \Vert{\w\vp_j}*f\Vert_\infty \le C\Vert f\Vert_\infty,\quad\text{for all}\quad f\in B,\quad  j\in\N,
\end{equation}
then, for any $f\in C(\T^d)$, the error estimate~\eqref{KS000B} can be improved in the following way
$$
\Vert f - Q_j(f,\phi_j,\w\phi_j) \Vert_\infty\le C \(\Vert {\psi_j}*T_j\Vert_{\infty}+K_{\phi_j,1} E_{\d M^j}(f)_\infty\).
$$
This estimate can be proved using the same argument as in the proof of Lemma~\ref{thKK}.

Note also that condition~\eqref{condinf} holds if, for example, $\w\phi_j$ is the periodic Dirac-delta function for all $j\in \N$.
\end{remark}

Finally, we note that combining Lemma~\ref{thKS+B} with Theorems~\ref{cor1}--\ref{thfr1b}, we easily obtain the following
error estimates given in terms of the unrestricted best approximation. Note also that inequality~\eqref{bbeell} below was earlier obtained in~\cite{KKS2}.

\begin{proposition}

Let $1\le p\le\infty$, $1/p+1/q=1$,  and $j\in \N$. Suppose that $\w\phi_j\in \mathcal{D}_{N,j,p}'$,
$\phi_j \in \mathcal{T}_{M^j}$, and $f\in \mathbb{B}_{p,1}^{d/p+N}(M)$.

\begin{itemize}
  \item[1)] If condition~\eqref{sc} holds for some $\d\in (0,1]$, then
\begin{equation}\label{bbeell}
  \Vert f - Q_j(f,\phi_j,\w\phi_j) \Vert_p\le CK_{\phi_j,q}m^{-j(\frac1p+\frac Nd)}\sum_{\nu=j}^\infty m^{(\frac1p+\frac Nd)\nu} E_{\delta M^\nu}(f)_p.
\end{equation}

\item[2)] If conditions~\eqref{zvezda1} and~\eqref{zvezda1m} hold for some $\d\in (0,1/2)$ and $s\in \N$, then
\begin{equation*}\label{KS000+NNNN++new}
\begin{split}
\Vert f - Q_j(f,\phi_j,\w\phi_j) \Vert_p\le C\bigg(\Omega_s(f,M^{-j})_p+K_{\phi_j,q} m^{-j(\frac1p+\frac Nd)}\sum_{\nu=j}^\infty m^{(\frac1p+\frac Nd)\nu} E_{\delta M^\nu}(f)_p\bigg).
\end{split}
\end{equation*}

\item[3)] If condition~\eqref{fr1} holds for some $\d\in (0,1/2)$ and $s>0$, then

\begin{equation*}\label{fr2new}
\begin{split}
\Vert f - Q_j(f,\phi_j,\w\phi_j) \Vert_p\le C\bigg(\mathcal{K}_s^\Delta(f,M^{-j})_p+K_{\phi_j,q} m^{-j(\frac1p+\frac Nd)}\sum_{\nu=j}^\infty m^{(\frac1p+\frac Nd)\nu} E_{\delta M^\nu}(f)_p\bigg).
\end{split}
\end{equation*}

\item[4)] If condition~\eqref{fr1b} holds for some $\d\in (0,1/2)$ and $s>0$, then

\begin{equation*}\label{fr2bsnew}
\begin{split}
\mathcal{K}_s^\Delta(f,M^{-j})_p\le C\bigg(\Vert f - Q_j(f,&\phi_j,\w\phi_j) \Vert_p\\
&+K_{\phi_j,q} m^{-j(\frac1p+\frac Nd)}\sum_{\nu=j}^\infty m^{(\frac1p+\frac Nd)\nu} E_{\frac12 M^\nu}(f)_p\bigg).
\end{split}
\end{equation*}
\end{itemize}

In the above four inequalities, the constant $C$ does not depend on $f$ and $j$.
\end{proposition}

\end{document}